\documentclass[11pt]{amsart}
\usepackage{amsfonts, amssymb, amsmath, amsthm, hyperref, color, float,enumerate}
\usepackage{url}

\usepackage[a4paper, left=2cm, right=2cm, top=3cm, bottom=3cm]{geometry} 

\theoremstyle{plain}
   \newtheorem{teo}{Theorem}
   \newtheorem{coro}[teo]{Corollary}
   \newtheorem{lema}[teo]{Lemma}

\theoremstyle{definition}
   
\theoremstyle{remark}

 \newtheorem{afirmacion}{Claim}

\numberwithin{equation}{section}
\allowdisplaybreaks

\definecolor{celestialblue}{rgb}{0.29, 0.59, 0.82}
\definecolor{chocolate(web)}{rgb}{0.82, 0.41, 0.12}

\hypersetup{
	colorlinks = true,
	linkcolor = celestialblue,
	anchorcolor = blue,
	citecolor = magenta,
	filecolor = blue,
	urlcolor = chocolate(web)
}

\begin{document}
	
	\title[Corrigendum]{Corrigendum to ``From $A_1$ to $A_\infty$: new mixed inequalities for certain maximal operators''}

	\author[F. Berra]{Fabio Berra}
	\address{CONICET and Departamento de Matem\'{a}tica (FIQ-UNL),  Santa Fe, Argentina.}
	\email{fberra@santafe-conicet.gov.ar}
	

		
	\thanks{The author was supported by CONICET and UNL}
	
	
	
	
	\begin{abstract}
		We devote this note to correct an estimate concerning mixed inequalities for the generalized maximal function $M_\Phi$, when certain properties of the associated  Young function $\Phi$ are assumed. 
		
		Although the obtained estimates turn out to be slightly different, they are good extensions of mixed inequalities for the classical Hardy-Littlewood maximal functions $M_r$, with $r\geq 1$. They also allow us to obtain mixed estimates for the generalized fractional maximal operator $M_{\gamma,\Phi}$, when $0<\gamma<n$ and $\Phi$ is an $L\log L$ type function.   
	\end{abstract}

	\maketitle
	
	\section*{Overview}
	
	Throughout this note we shall consider a Young function $\Phi$ with the following properties. Given $r\geq 1$ and $\delta\geq 0$,  we say that a Young function $\Phi$ belongs to the family $\mathfrak{F}_{r,\delta}$ if it is submultiplicative, has lower type $r$ and satisfies the condition
	\[\frac{\Phi(t)}{t^r}\leq C_0 (\log t)^\delta,\quad \textrm{ for }t\geq t^*,\]
	for some constants $C_0>0$ and $t^*\geq 1$.
	
	In \cite{Berra-Carena-Pradolini(MN)} we obtained mixed estimates for the operator $M_\Phi$, where $\Phi$ belongs to $\mathfrak{F}_{r,\delta}$. Concretely, we stated the inequality 
	\[uv^r\left(\left\{x\in \mathbb{R}^n: \frac{M_\Phi(fv)(x)}{v(x)}>t
\right\}\right)\leq C\int_{\mathbb{R}^n}\Phi\left(\frac{|f|}{t}\right)uv^r\]
where $u$ and $v^r$ are weights belonging to the $A_1$-Muckenhoupt class.

Later, in \cite{B22Pot}, the same kind of estimate was obtained when $v^r$ is only assumed to be an $A_\infty$ weight.

In the the proofs of both results we used Claim 3.4 in \cite{Berra-Carena-Pradolini(MN)}, and Claims 1 and 3 in \cite{B22Pot} as auxiliary tools. These claims have an error on a Hölder estimate, where a limiting argument was mistakenly used and it cannot be adapted to obtain the inequality given above. 

The purpose of this note is give a proof of Theorem~1 in \cite{B22Pot} that avoids this step on the claims and allows us to obtain a slightly different estimate, that will still be useful for our purposes. We shall only modify the results obtained in \cite{B22Pot}, since they are more general and the corresponding version of those in \cite{Berra-Carena-Pradolini(MN)} will follow as an immediate consequence. The modified mixed estimate in \cite{B22Pot} is the following.
	
	\begin{teo}[Corrected version of Theorem 1 in \cite{B22Pot}]\label{teo: teorema principal}
		Let $r\geq 1$, $\delta\geq 0$ and $\Phi\in \mathfrak{F}_{r,\delta}$. If $u\in A_1$ and $v^r\in A_\infty$ then there exists a positive constant $\varepsilon_0$  such that the inequality
		\[uv^r\left(\left\{x\in \mathbb{R}^n: \frac{M_\Phi(fv)(x)}{v(x)}>t
		\right\}\right)\leq C\int_{\mathbb{R}^n}\left(\eta_\varepsilon\circ \Phi\right)\left(\frac{|f(x)|}{t}\right)u(x)v^r(x)\,dx\]
		holds for every positive $t$ and every $0<\varepsilon<\varepsilon_0$, where  $\eta_\varepsilon(z)=z(1+\log^+z)^{\delta/\varepsilon}$ and $C$ depends on $\varepsilon$.
	\end{teo}
	
	It is not difficult to see that $M_\Phi v\gtrsim v$ when $\Phi$ belongs to $\mathfrak{F}_{r,\delta}$. So we have the following result as an immediate consequence of the theorem above.
	
	\begin{coro}[Corrected version of Corollary 2 in \cite{B22Pot}]\label{coro: corolario del teorema principal - 1}
Under the assumptions in Theorem~\ref{teo: teorema principal}, there exists a positive constant $\varepsilon_0$ such that
\[uv^r\left(\left\{x\in \mathbb{R}^n: \frac{M_\Phi(fv)(x)}{M_\Phi v(x)}>t
	\right\}\right)\leq C\int_{\mathbb{R}^n}\left(\eta_\varepsilon\circ \Phi\right)\left(\frac{|f(x)|}{t}\right)u(x)v^r(x)\,dx\]
	holds for every positive $t$ and every $0<\varepsilon<\varepsilon_0$, where  $\eta_\varepsilon$ is as above and $C$ depends on $\varepsilon$.
\end{coro}

Throughout these notes, all references, lemmas and theorems will follow the label given in \cite{B22Pot}.

	\section{Proof of Theorem~\ref{teo: teorema principal}}

	We shall first give some preliminaries in order to proceed with the proof. Recall that $w\in A_\infty$ if there exists a positive constant $C$ such that
	\[\left(\frac{1}{|Q|}\int_Q w\right)\text{exp}\left(\frac{1}{|Q|}\int_Q \log w^{-1}\right)\leq C\]
	for every cube $Q$ in $\mathbb{R}^n$. The smallest constant for which the inequality above holds is denoted by $[w]_{A_\infty}$.
	
	The following lemma will be useful in the sequel. It can be found in \cite{HP13}.
	
	\begin{lema}\label{lema: reverse Holder con constante}
	Let $w\in A_\infty$ and let $r_w=1+\frac{1}{\tau_n[w]_{A_\infty}}$. Then for any cube $Q$ we have
	\[\left(\frac{1}{|Q|}\int_Q w^{r_w}\right)^{1/r_w}\leq \frac{2}{|Q|}\int_Q w.\]
	As a consequence, given any cube $Q$ and a measurable set $E\subseteq Q$ we have that
	\[\frac{w(E)}{w(Q)}\leq 2\left(\frac{|E|}{|Q|}\right)^{\varepsilon_w},\]
	where $\varepsilon_w=1/(1+\tau_n[w]_{A_\infty})$. The constant $\tau_n$ is purely dimensional and can be chosen as $2^{11+n}$.
	\end{lema}

	Recall that we are dealing with a function $\Phi\in \mathfrak{F}_{r,\delta}$, where $r\geq 1$ and $\delta\geq 0$ are given. Since we are assuming $v^r\in A_\infty$, there exists $\varepsilon_1>0$ such that $v^{r+\varepsilon}\in A_\infty$ for every $0<\varepsilon\leq \varepsilon_1$.
	
	Fix $0<\varepsilon<\min\{\varepsilon_1,\varepsilon_2\}$, where $\varepsilon_2>0$ will be chosen later. Then we have that $v^r\in \text{RH}_s$, where $s=1+\varepsilon/r$. We shall denote $\Psi_\varepsilon=\eta_\varepsilon \circ \Phi$. We shall follow the same sketch and steps as in \cite{B22Pot}, where the entire proof is included for the sake of clearness. Recall that
 it is enough to prove that
	\[uv^r\left(\left\{x\in \mathbb{R}^n: \frac{M_{\Phi,\mathcal{D}}(fv)(x)}{v(x)}>t
\right\}\right)\leq C_\varepsilon\int_{\mathbb{R}^n}\Psi_\varepsilon\left(\frac{|f(x)|}{t}\right)u(x)v^r(x)\,dx,\]
where $\mathcal{D}$ is a given dyadic grid.
We can also assume that $t=1$ and that $g=|f|v$ is a bounded function with compact support. Then, for a fixed number $a>2^n$, we can write
\begin{align*}
  uv^r\left(\left\{x\in \mathbb{R}^n: \frac{M_{\Phi,\mathcal{D}}(fv)(x)}{v(x)}>1
\right\}\right)&=\sum_{k\in \mathbb{Z}} uv^r\left(\left\{x: \frac{M_{\Phi,\mathcal{D}}g(x)}{v(x)}>1, a^k<v\leq a^{k+1}
\right\}\right)\\
&=:\sum_{k\in \mathbb{Z}}uv^r(E_k).  \end{align*}
For every $k\in \mathbb{Z}$ we consider the set
\[\Omega_k=\left\{x\in \mathbb{R}^n: M_{\Phi,\mathcal{D}}g(x)>a^k\right\},\]
and by virtue of the Calderón-Zygmund decomposition of the space (see \cite[Lemma~6]{B22Pot}) there exists a collection of disjoint dyadic cubes $\{Q_j^k\}_j$ that satisfies
\[\Omega_k=\bigcup_j Q_j^k,\]
and $\|g\|_{\Phi,Q_j^k}>a^k$ for each $j$. By maximality, we have
\begin{equation}\label{eq: promedios Luxemburgo de g son como a^k}
a^k<\|g\|_{\Phi,Q_j^k}\leq 2^n a^k, \quad \textrm{ for every }j.
\end{equation}
For every $k\in \mathbb{Z}$ we now proceed to split the obtained cubes in different classes, as in \cite{L-O-P}. Given a nonnegative integer  $\ell$, we set
\[\Lambda_{\ell,k}=\left\{Q_j^k: a^{(k+\ell)r}\leq \frac{1}{|Q_j^k|}\int_{Q_j^k} v^r< a^{(k+\ell+1)r}\right\},\]
and also
\[\Lambda_{-1,k}=\left\{Q_j^k:  \frac{1}{|Q_j^k|}\int_{Q_j^k} v^r< a^{kr}\right\}.\]

The next step is to split every cube in the family $\Lambda_{-1,k}$. Fixed $Q_j^k\in \Lambda_{-1,k}$, we perform the  Calder\'on-Zygmund decomposition of the function $v^r\mathcal{X}_{Q_j^k}$ at level $a^{kr}$. Then we obtain, for each $k$, a collection $\left\{Q_{j,i}^k\right\}_i$ of maximal cubes, contained in $Q_j^k$ and which satisfy
\begin{equation}\label{eq: promedios de v^r sobre Q_{j,i}^k son como a^{kr}}
a^{kr}<\frac{1}{|Q_{j,i}^k|}\int_{Q_{j,i}^k}v^r\leq 2^na^{kr},\quad \textrm{ for every }i.
\end{equation}

Also we define the sets
\[\Gamma_{\ell,k}=\left\{Q_j^k\in \Lambda_{\ell,k}: \left|Q_j^k\cap \left\{x: a^k<v\leq a^{k+1}\right\}\right|>0\right\},\]
and also
\[\Gamma_{-1,k}=\left\{Q_{j,i}^k: Q_j^k\in \Lambda_{-1,k} \textrm{ and } \left|Q_{j,i}^k\cap \left\{x: a^k<v\leq a^{k+1}\right\}\right|>0\right\}.\]

Since $E_k\subseteq \Omega_k$, we can estimate
\begin{align*}
\sum_{k\in \mathbb{Z}} uv^r(E_k)&=\sum_{k\in \mathbb{Z}} uv^r(E_k\cap \Omega_k)\\
&=\sum_{k\in \mathbb{Z}} \sum_j uv^r(E_k\cap Q_j^k)\\
&\leq \sum_{k\in \mathbb{Z}}\sum_{\ell\geq 0} \sum_{Q_j^k\in \Gamma_{\ell,k}}a^{(k+1)r}u(E_k\cap Q_j^k)+\sum_{k\in \mathbb{Z}}\,\, \sum_{i:Q_{j,i}^k\in \Gamma_{-1,k}}a^{(k+1)r}u(Q_{j,i}^k).
\end{align*}

If we can prove that given a negative integer $N$, there exists a positive constant $C_\varepsilon$, independent of $N$ for which the following estimate
\begin{equation}\label{eq: desigualdad con C independiente de N}
\sum_{k\geq N}\sum_{\ell\geq 0} \sum_{Q_j^k\in \Gamma_{\ell,k}}a^{(k+1)r}u(E_k\cap Q_j^k)+\sum_{k\geq N} \sum_{i:Q_{j,i}^k\in \Gamma_{-1,k}}a^{(k+1)r}u(Q_{j,i}^k)\leq C_\varepsilon\int_{\mathbb{R}^n}\Psi_\varepsilon\left(|f|\right)uv^r
\end{equation}
holds, then the proof would be completed by letting $N\to-\infty$.

We shall also need the following lemma from \cite{L-O-P}. We include an adaptation of the proof involving our parameters for the sake of clearness.

\begin{lema}\label{lema: condicion Ainf de u}
    Let $\ell\geq 0$ and $Q_j^k\in \Gamma_{\ell,k}$. If $u\in A_\infty$ and $v^r\in A_q$ for some $1<q<\infty$, then there exists positive constants $c_1$ and $c_2$ depending on $u$ and $v^r$ such that
    \[u(E_k\cap Q_j^k)\leq c_1\,e^{-c_2r\ell}u(Q_j^k).\]
    Furthermore, we can pick $c_1=2\left([v^r]_{A_q}a^r\right)^{1/((q-1)(1+\tau_n[u]_{A_\infty}))}$ and $c_2=\ln a/((q-1)(1+\tau_n[u]_{A_\infty}))$, where $\tau_n$ is the dimensional constant appearing in Lemma~\ref{lema: reverse Holder con constante}.
\end{lema}

\begin{proof}
    Since $v^r\in A_\infty$, there exists $q>1$ such that $v^r\in A_q$. Since $Q_j^k\in \Gamma_{\ell,k}$, we have that
    \[\left(\frac{|E_k\cap Q_j^k|}{|Q_j^k|}\right)^{q-1}\leq \left(\frac{1}{|Q_j^k|}\int_{Q_j^K} v^{r(1-q')}\right)^{q-1}a^{r(k+1)}\leq \frac{[v^r]_{A_q}|Q_j^k|}{v^r(Q_j^k)}a^{r(k+1)}\leq [v^r]_{A_q}a^{(1-\ell)r}.\]
    Since $u\in A_1\subseteq A_\infty$, by Lemma~\ref{lema: reverse Holder con constante} and the estimate above we have that
    \[\frac{u(E_k\cap Q_j^k)}{u(Q_j^k)}\leq 2\left(\frac{|E_k\cap Q_j^k|}{|Q_j^k|}\right)^{1/(1+\tau_n[u]_{A_\infty})}\leq 2 \left([v^r]_{A_q}a^{(1-\ell)r}\right)^{1/((q-1)(1+\tau_n[u]_{A_\infty}))}.\]
    From this last inequality we can obtain the thesis.
\end{proof}

\begin{proof}[Proof of Theorem~\ref{teo: teorema principal}]
 Since $u\in A_1$, we have $u\in A_\infty$. Moreover, the assumption $v^r\in A_\infty$ implies that there exists $1<q<\infty$ such that $v^r\in A_q$. We take $\varepsilon_2=r/((q-1)(1+\tau_n[u]_{A_\infty}))$ and $\varepsilon_0=\min\{\varepsilon_1,\varepsilon_2\}$. Fixed $0<\varepsilon<\varepsilon_0$, recall that we have to estimate the two quantities
	\[A_N:= \sum_{k\geq N}\sum_{\ell\geq 0} \sum_{Q_j^k\in \Gamma_{\ell,k}}a^{(k+1)r}u(E_k\cap Q_j^k)\]
	and
	\[B_N:=\sum_{k\geq N} \sum_{i:Q_{j,i}^k\in \Gamma_{-1,k}}a^{(k+1)r}u(Q_{j,i}^k)\]
	by $C_\varepsilon\int_{\mathbb{R}^n}\Psi_\varepsilon\left(|f|\right)uv^r$, with $C_\varepsilon$ independent of $N$.
	
	We shall start with the estimate of $A_N$. Fix $\ell\geq 0$ and let $\Delta_\ell=\bigcup_{k\geq N} \Gamma_{\ell,k}$. We define recursively a sequence of sets as follows: 
	\[P_0^\ell=\{Q: Q \textrm{ is maximal in } \Delta_\ell \textrm{ in the sense of inclusion}\}\]
	and for $m\geq 0$ given we say that $Q_j^k\in P_{m+1}^\ell$ if there exists a cube $Q_s^t$ in $P_m^\ell$ which verifies
	\begin{equation}\label{eq: desigualdad 1 conjunto P_m^l}
	\frac{1}{|Q_j^k|}\int_{Q_j^k} u>\frac{2}{|Q_s^t|}\int_{Q_s^t}u
	\end{equation}
	and it is maximal in this sense, that is,
	\begin{equation}\label{eq: desigualdad 2 conjunto P_m^l}
	\frac{1}{|Q_{j'}^{k'}|}\int_{Q_j^k} u\leq\frac{2}{|Q_s^t|}\int_{Q_s^t}u
	\end{equation}
	for every $Q_j^k\subsetneq Q_{j'}^{k'}\subsetneq Q_s^t$.
	
	Let $P^\ell=\bigcup_{m\geq 0} P_m^{\ell}$, the set of principal cubes in $\Delta_\ell$. By applying Lemma~\ref{lema: condicion Ainf de u} and the definition of $\Lambda_{\ell,k}$ we have that
	\begin{align*}
	\sum_{k\geq N}\sum_{\ell\geq 0} \sum_{Q_j^k\in \Gamma_{\ell,k}}a^{(k+1)r}u(E_k\cap Q_j^k)&\leq\sum_{k\geq N}\sum_{\ell\geq 0} \sum_{Q_j^k\in \Gamma_{\ell,k}}c_1a^{(k+1)r}e^{-c_2\ell r}u(Q_j^k)\\
	&\leq \sum_{\ell\geq 0}c_1e^{-c_2\ell r}a^{r(1-\ell)}\sum_{k\geq N}\sum_{Q_j^k\in \Gamma_{\ell,k}}\frac{v^r(Q_j^k)}{|Q_j^k|}u(Q_j^k).
	\end{align*}
	
	Let us sort the inner double sum in a more convenient way. We define
	\[\mathcal{A}_{(t,s)}^\ell=\left\{Q_j^k \in \bigcup_{k\geq N} \Gamma_{\ell,k}: Q_j^k\subseteq Q_s^t \textrm{ and } Q_s^t \textrm{ is the smallest cube in }P^\ell \textrm{ that contains it} \right\}.\]
	That is, every $Q_j^k\in \mathcal{A}_{(t,s)}^\ell$ is not a principal cube, unless $Q_j^k=Q_s^t$. Recall that $v^r\in A_\infty$ implies that there exist two positive constants $C$ and $\theta$ verifying
	\begin{equation}\label{eq: condicion Ainfty de v^r}
	\frac{v^r(E)}{v^r(Q)}\leq C\left(\frac{|E|}{|Q|}\right)^{\theta},
	\end{equation}
	for every cube $Q$ and every measurable set $E$ of $Q$.
	
	By using \eqref{eq: desigualdad 2 conjunto P_m^l} and Lemma~12 in \cite{B22Pot} we have that
	\begin{align*}
	\sum_{k\geq N}\sum_{Q_j^k\in \Gamma_{\ell,k}}\frac{v^r(Q_j^k)}{|Q_j^k|}u(Q_j^k)&=\sum_{Q_s^t \in P^\ell}\,\,\sum_{(k,j): Q_j^k\in \mathcal{A}_{(t,s)}^\ell}\frac{u(Q_j^k)}{|Q_j^k|}v^r(Q_j^k)\\
	&\leq 2\sum_{Q_s^t \in P^\ell}\frac{u(Q_s^t)}{|Q_s^t|}\,\,\sum_{(k,j): Q_j^k\in \mathcal{A}_{(t,s)}^\ell}v^r(Q_j^k)\\
	&\leq C\sum_{Q_s^t \in P^\ell}\frac{u(Q_s^t)}{|Q_s^t|}v^r(Q_s^t)\left(\frac{\left|\bigcup_{(k,j): Q_j^k\in \mathcal{A}_{(t,s)}^\ell}Q_j^k\right|}{|Q_s^t|}\right)^\theta\\
	&\leq C\sum_{Q_s^t \in P^\ell} \frac{u(Q_s^t)}{|Q_s^t|}v^r(Q_s^t).
	\end{align*}
	
	Therefore,
	\begin{align*}
	\sum_{k\geq N}\sum_{\ell\geq 0} \sum_{Q_j^k\in \Gamma_{\ell,k}}a^{(k+1)r}u(E_k\cap Q_j^k)&\leq C\sum_{\ell\geq 0}e^{-c_2\ell r}a^{-\ell r}\sum_{Q_s^t \in P^\ell} \frac{v^r(Q_s^t)}{|Q_s^t|}u(Q_s^t)\\
	&\leq C\sum_{\ell\geq 0}e^{-c_2\ell r}\sum_{Q_s^t \in P^\ell} a^{tr}u(Q_s^t).
	\end{align*}
	
	\begin{afirmacion}[Corrected version of Claim 1 in \cite{B22Pot}]\label{af: claim 1}
	Given $\ell\geq 0$ and $Q_j^k\in \bigcup_{k\geq N} \Gamma_{\ell,k}$, we have that
		\begin{equation}\label{eq: estimacion de afirmacion: control de akr por promedios de Phi(f). caso l no negativo}
		a^{kr}\leq C\frac{\ell^{\delta/\varepsilon}a^{\ell\varepsilon}}{|Q_j^k|}\int_{Q_j^k}\Psi_\varepsilon\left(|f(x)|\right)v^r(x)\,dx,
		\end{equation}
		where $C$ depends on $\varepsilon$. 
	\end{afirmacion}
	By applying the estimate above we obtain
	\begin{align*}
	\sum_{k\geq N}\sum_{\ell\geq 0} \sum_{Q_j^k\in \Gamma_{\ell,k}}a^{(k+1)r}u(E_k\cap Q_j^k)&\leq C_\varepsilon\sum_{\ell\geq 0} e^{-c_2\ell r}\ell^{\delta/\varepsilon}a^{\ell\varepsilon}\sum_{Q_s^t\in P^{\ell}}\frac{u(Q_s^t)}{|Q_s^t|}\int_{Q_s^t}\Psi_\varepsilon\left(|f|\right)v^r\\
	&=C_\varepsilon\sum_{\ell\geq 0}e^{-c_2\ell r}\ell^{\delta/\varepsilon}a^{\ell\varepsilon}\int_{\mathbb{R}^n} \Psi_\varepsilon\left(|f|\right)v^r\left(\sum_{Q_s^t\in P^{\ell}}\frac{u(Q_s^t)}{|Q_s^t|}\mathcal{X}_{Q_s^t}\right)\\
	&=C_\varepsilon\sum_{\ell\geq 0}e^{-c_2\ell r}\ell^{\delta/\epsilon}a^{\ell\varepsilon}\int_{\mathbb{R}^n} \Psi_\varepsilon\left(|f(x)|\right)v^r(x)h_1(x)\,dx\\
	&\leq C_\varepsilon\int_{\mathbb{R}^n} \Psi_\varepsilon\left(|f(x)|\right)v^r(x)u(x)\,dx,
	\end{align*}
	by virtue of Claim 2 in \cite{B22Pot}. Notice that the sum is finite since we are assuming $\varepsilon<\varepsilon_2$. Indeed, we have that
	\[e^{-c_2\ell r}a^{\ell \varepsilon}=e^{-c_2\ell r+\ell \varepsilon\ln a}=e^{\ell(-c_2r+\varepsilon \ln a)},\]
	and this exponent is negative by the election of $\varepsilon$. This completes the estimate of $A_N$.
	
	\medskip
	
		Let us center our attention on the estimate of $B_N$. Fix $0<\beta<\theta$, where $\theta$ is the number appearing in \eqref{eq: condicion Ainfty de v^r}. We shall build the set of principal cubes in $\Delta_{-1}=\bigcup_{k\geq N}\Gamma_{-1,k}$. Let
	\[P_0^{-1}=\{Q: Q \textrm{ is a maximal cube in }\Delta_{-1}\textrm{ in the sense of inclusion}\}\]
	and, recursively, we say that $Q_{j,i}^k\in P_{m+1}^{-1}$, $m\geq 0$, if there exists a cube $Q_{s,l}^t\in P_m^{-1}$ such that
	\begin{equation}\label{eq: desigualdad 1 conjunto P_m^{-1}}
	\frac{1}{|Q_{j,i}^k|}\int_{Q_{j,i}^k} u> \frac{a^{(k-t)\beta r}}{|Q_{s,l}^t|}\int_{Q_{s,l}^t}u
	\end{equation} 
	and it is the biggest subcube of $Q_{s,l}^t$ that verifies this condition, that is
	\begin{equation}\label{eq: desigualdad 2 conjunto P_m^{-1}}
	\frac{1}{|Q_{j',i'}^{k'}|}\int_{Q_{j',i'}^{k'}} u\leq \frac{a^{(k-t)\beta r}}{|Q_{s,l}^t|}\int_{Q_{s,l}^t}u
	\end{equation} 
	if $Q_{j,i}^k\subsetneq Q_{j',i'}^{k'}\subsetneq Q_{s,l}^t$. Let $P^{-1}=\bigcup_{m\geq 0} P_m^{-1}$, the set of principal cubes in $\Delta_{-1}$. Similarly as before, we define the set
	\[\mathcal{A}_{(t,s,l)}^{-1}=\left\{Q_{j,i}^k \in \bigcup_{k\geq N} \Gamma_{-1,k}: Q_{j,i}^k\subseteq Q_{s,l}^t \textrm{ and } Q_{s,l}^t \textrm{ is the smallest cube in }P^{-1} \textrm{ that contains it} \right\}.\]
	
	We can therefore estimate $B_N$ as follows
	\begin{align*}
	B_N&\leq a^r \sum_{k\geq N}\sum_{i:Q_{j,i}^k\in \Gamma_{-1,k}}\frac{v^r(Q_{j,i}^k)}{|Q_{j,i}^k|}u(Q_{j,i}^k)\\
	&\leq a^r\sum_{Q_{s.l}^t \in P^{-1}}\sum_{k,j,i: Q_{j,i}^k\in \mathcal{A}_{(t,s,l)}^{-1}}\frac{u(Q_{j,i}^k)}{|Q_{j,i}^k|}v^r(Q_{j,i}^k)\\
	&\leq a^r\sum_{Q_{s.l}^t \in P^{-1}}\frac{u(Q_{s,l}^t)}{|Q_{s,l}^t|}\sum_{k\geq t}a^{(k-t)\beta r}\,\sum_{j,i: Q_{j,i}^k\in \mathcal{A}_{(t,s,l)}^{-1}}v^r(Q_{j,i}^k).
	\end{align*}
	
	Fixed $k\geq t$, observe that
	\[\sum_{j,i: Q_{j,i}^k\in \mathcal{A}_{(t,s,l)}^{-1}} |Q_{j,i}^k|<\sum_{j,i: Q_{j,i}^k\in \mathcal{A}_{(t,s,l)}^{-1}} a^{-kr}v^r(Q_{j,i}^k)\leq a^{-kr}v^r(Q_{s,l}^t)\leq 2^na^{(t-k)r}|Q_{s,l}^t|.\]
	Combining this inequality with the $A_\infty$ condition of $v^r$ we have, for every $k\geq t$, that
	\begin{align*}
	\sum_{j,i: Q_{j,i}^k\in \mathcal{A}_{(t,s,l)}^{-1}}a^{(k-t)\beta r}v^r(Q_{j,i}^k)&\leq Cv^r(Q_{s,l}^t)\left(\frac{\sum_{j,i: Q_{j,i}^k\in \mathcal{A}_{(t,s,l)}^{-1}}|Q_{j,i}^k|}{|Q_{s,l}^t|}\right)^\theta\\
	&\leq Ca^{(t-k)r\theta}.
	\end{align*}
	Thus,
	\begin{align*}
	B_N&\leq C\sum_{Q_{s.l}^t \in P^{-1}}\frac{u(Q_{s,l}^t)}{|Q_{s,l}^t|}v^r(Q_{s,l}^t)\sum_{k\geq t} a^{(t-k)r(\theta-\beta)}\\
	&=C\sum_{Q_{s.l}^t \in P^{-1}}\frac{v^r(Q_{s,l}^t)}{|Q_{s,l}^t|}u(Q_{s,l}^t)\\
	&\leq C\sum_{Q_{s.l}^t \in P^{-1}}a^{tr}u(Q_{s,l}^t).
	\end{align*}
	\begin{afirmacion}[Corrected version of Claim 3 in \cite{B22Pot}]\label{af: Claim 3}
If $Q_{j}^k\in \Lambda_{-1,k}$ then there exists a positive constant $C_\varepsilon$ such that
		\[a^{kr}\leq \frac{C_\varepsilon}{|Q_j^k|}\int_{Q_j^k}\Psi_\varepsilon\left(|f(x)|\right)v^r(x)\,dx.\]
	\end{afirmacion}
	
	By using this estimate we can proceed as follows
	\begin{align*}
	\sum_{k\geq N}\sum_{i:Q_{j,i}^k\in \Gamma_{-1,k}} a^{(k+1)r}u(Q_{j,i}^k)&\leq C\sum_{Q_{s.l}^t \in P^{-1}}a^{tr}u(Q_{s,l}^t)\\
	&\leq C_\varepsilon\sum_{Q_{s.l}^t \in P^{-1}}\frac{u(Q_{s,l}^t)}{|Q_s^t|}\int_{Q_s^t}\Psi_\varepsilon\left(|f(x)|\right)v^r(x)\,dx\\
	&\leq C_\varepsilon\int_{\mathbb{R}^n} \Psi_\varepsilon\left(|f(x)|\right)v^r(x)\left[\sum_{Q_{s.l}^t \in P^{-1}}\frac{u(Q_{s,l}^t)}{|Q_s^t|}\mathcal{X}_{Q_s^t}(x)\right]\,dx\\
	&=C_\varepsilon\int_{\mathbb{R}^n} \Psi_\varepsilon\left(|f(x)|\right)v^r(x)h_2(x)\,dx\\
	&\leq C_\varepsilon\int_{\mathbb{R}^n} \Psi_\varepsilon\left(|f(x)|\right)u(x)v^r(x)\,dx,
	\end{align*}
	by virtue of Claim~4 in \cite{B22Pot}. This concludes the proof.
\end{proof}

We proceed with the proofs of the claims, in order to complete the argument above.

\begin{proof}[Proof of Claim~\ref{af: claim 1}]
	Fix $\ell\geq 0$ and a cube $Q_j^k\in \bigcup_{k\geq N} \Gamma_{\ell,k}$. We know that $\|g\|_{\Phi,Q_j^k}>a^k$ or,  equivalently, $\left\|\frac{g}{a^k}\right\|_{\Phi, Q_j^k}>1$. Denote with $A=\{x\in Q_j^k: v(x)\leq t^*a^k\}$ and $B=Q_j^k\backslash A$, where $t^*$ is the number verifying that if $z\geq t^*$, then
\[\frac{\Phi(z)}{z^r}\leq C_0 \left(\log z\right)^\delta.\]
Then,
\[1<\left\|\frac{g}{a^k}\right\|_{\Phi, Q_j^k}\leq \left\|\frac{g}{a^k}\mathcal{X}_A\right\|_{\Phi, Q_j^k}+\left\|\frac{g}{a^k}\mathcal{X}_B\right\|_{\Phi, Q_j^k}=I+II.\]
This inequality implies that either $I>1/2$ or $II>1/2$. Since $\Phi\in\mathfrak{F}_{r,\delta}$ we can easily see that $I>1/2$ implies that
\[a^{kr}<\frac{2^rC_0(\log (2t^*))^{\delta}}{|Q_j^k|}\int_{Q_j^k}\Phi\left(|f|\right)v^r\leq \frac{2^rC_0(\log (2t^*))^{\delta}}{|Q_j^k|}\int_{Q_j^k}(\eta_\varepsilon \circ \Phi)\left(|f|\right)v^r,\]
because $\eta_\varepsilon(z)\geq z$.

On the other hand, if $II>1/2$ then again 
\begin{align*}
1&<\frac{1}{|Q_j^k|}\int_B\Phi\left(\frac{2|f|v}{a^k}\right)\\
&\leq \frac{\Phi(2)C_0}{|Q_j^k|}\int_B \Phi\left(|f|\right)\frac{v^r}{a^{kr}}\left(\log\left(\frac{v}{a^k}\right)\right)^{\delta},
\end{align*}
since $\Phi\in \mathfrak{F}_{r,\delta}$. This implies that
\[a^{kr}\leq \frac{\Phi(2)C_0}{|Q_j^k|}\int_{Q_j^k}\Phi\left(|f|\right)v^rw_k,\]
where $w_k(x)=\left(\log\left(\frac{v(x)}{a^k}\right)\right)^{\delta}\mathcal{X}_{B}(x)$. We shall now perform a generalized Hölder inequality with the Young functions
\[\eta_\varepsilon(z)=z(1+\log^+z)^{\delta/\varepsilon}\quad \text{ and }\quad \tilde\eta_\varepsilon(z) \approx (e^{z^{\varepsilon/\delta}}-e)\mathcal{X}_{(1,\infty)}(z),\]
 with respect to the measure $d\mu(x)=v^r(x)\,dx$. Thus we have
\begin{equation}\label{eq: af: Claim 1 - eq1}
\frac{1}{|Q_j^k|}\int_{Q_j^k}\Phi\left(|f|\right)w_kv^r\leq \frac{v^r(Q_j^k)}{|Q_j^k|}\|\Phi(|f|)\|_{\eta_\varepsilon,Q_j^k,v^r}\|w_k\|_{\tilde\eta_\varepsilon,Q_j^k,v^r}.    
\end{equation}
Let us first estimate the last factor.  Since $e^{(\log z)^\varepsilon}\leq z^\varepsilon$ when $z\geq e^{\varepsilon^{1/(\varepsilon-1)}}$, we proceed as follows 
\begin{equation}\label{eq: af: Claim 1 - eq2}
    \frac{1}{v^r(Q_j^k)}\int_{Q_j^k}\tilde\eta_\varepsilon(w_k)v^r\leq \tilde\eta_\varepsilon\left(\varepsilon^{\delta/(\varepsilon-1)}\right)+\frac{1}{v^r(Q_j^k)}\int_{Q_j^k\cap\left\{v/a^k>e^{\varepsilon^{1/(\varepsilon-1)}}\right\}}\frac{v^{r+\varepsilon}}{a^{k\varepsilon}}.
\end{equation}
Since $\varepsilon^{-\varepsilon}\leq e$ we have that
\[\tilde\eta_\varepsilon\left(\varepsilon^{\delta/(\varepsilon-1)}\right)\leq e^{e^{1/(1-\varepsilon)}}.\]
On the other hand, our hypothesis on $v$ implies that $v^r\in \text{RH}_s$, where $s=1+\varepsilon/r$. Since $v^{r}\in \Lambda_{\ell,k}$ we obtain
\begin{align*}
\frac{1}{v^r(Q_j^k)}\int_{Q_j^k\cap\left\{v/a^k>e^{\varepsilon^{1/(\varepsilon-1)}}\right\}}\frac{v^{r+\varepsilon}}{a^{k\varepsilon}}&\leq  \frac{a^{-k\varepsilon}}{v^r(Q_j^k)}\int_{Q_j^k}v^{r+\varepsilon}\\
&\leq [v^r]_{\text{RH}_{s}}^{s}\frac{a^{-k\varepsilon}|Q_j^k|}{v^r(Q_j^k)}\left(\frac{1}{|Q_j^k|}\int_{Q_j^k}v^r\right)^{s}\\
&= [v^r]_{\text{RH}_{s}}^{s} a^{-k\varepsilon}\left(\frac{1}{|Q_j^k|}\int_{Q_j^k}v^r\right)^{s-1}\\
&\leq [v^r]_{\text{RH}_{s}}^{s} a^{-k\varepsilon}a^{(k+\ell+1)\varepsilon}\\
&=[v^r]_{\text{RH}_{s}}^{s}a^{(\ell+1)\varepsilon}.   
\end{align*}
By using these two estimates in \eqref{eq: af: Claim 1 - eq2}, we get
\[\|w_k\|_{\tilde\eta_\varepsilon,Q_j^k,v^r}\leq e^{e^{1/(1-\varepsilon)}}+ [v^r]_{\text{RH}_{s}}^{s}a^{(\ell+1)\varepsilon}\leq (e^{e^{1/(1-\varepsilon)}}+[v^r]_{\text{RH}_{s}}^{s})a^{(\ell+1)\varepsilon}.\]
 
We also observe that
\begin{equation}\label{eq: relacion norma con infimo}
  \|\Phi(|f|)\|_{\eta_\varepsilon,Q_j^k,v^r}\approx \inf_{\tau>0}\left\{\tau+\frac{\tau}{v^r(Q_j^k)}\int_{Q_j^k}\eta_\varepsilon\left(\frac{\Phi(|f|)}{\tau}\right)v^r\right\}.  
\end{equation}
If we choose $\tau=(2a^{(\ell+1)(r+\varepsilon)}(e^{e^{1/(1-\varepsilon)}}+[v^r]_{\text{RH}_{s}}^{s}))^{-1}$ then we can estimate the right-hand side of \eqref{eq: af: Claim 1 - eq1} as follows
\begin{align*}
    \frac{v^r(Q_j^k)}{|Q_j^k|}\|\Phi(|f|)\|_{\eta_\varepsilon,Q_j^k,v^r}\|w_k\|_{\tilde\eta_\varepsilon,Q_j^k,v^r}&\leq \frac{a^{kr}}{2}+(e^{e^{1/(1-\varepsilon)}}+[v^r]_{\text{RH}_{s}}^{s})a^{(\ell+1)\varepsilon}\tau\eta_\varepsilon\left(\frac{1}{\tau}\right)\frac{1}{|Q_j^k|}\int_{Q_j^k}\Psi_\varepsilon(|f|)v^r.
\end{align*}
Notice that
\[\tau\eta_\varepsilon\left(\frac{1}{\tau}\right)=\left(1+\log\left(\frac{1}{\tau}\right)\right)^{\delta/\varepsilon}\leq 2^{\delta/\varepsilon}\left(\log(2(e^{e^{1/(1-\varepsilon)}}+[v^r]_{\text{RH}_{s}}^{s}))+(\ell+1)(r+\varepsilon)\log a\right)^{\delta/\varepsilon}\leq C_\varepsilon \ell^{\delta/\varepsilon}.\]
By plugging these two estimates in \eqref{eq: af: Claim 1 - eq1} we arrive to
\begin{align*}
    a^{kr}&\leq \frac{C_\varepsilon \ell^{\delta/\varepsilon}a^{\ell\varepsilon}}{|Q_j^k|}\int_{Q_j^k}\Psi_\varepsilon(|f|)v^r,
\end{align*}
and we are done.
	\end{proof}
	
	\medskip
	
\begin{proof}[Proof of Claim~\ref{af: Claim 3}]
	The proof follows similar arguments as the previous one. By adopting the same notation, we have that $\left\|\frac{g}{a^k}\right\|_{\Phi, Q_j^k}>1$, and this implies that either $I>1/2$ or $II>1/2$. If $I>1/2$, we obtain 
	\[a^{kr}<\frac{C_0(\log (2t^*))^{\delta}}{|Q_j^k|}\int_{Q_j^k}\Phi\left(|f|\right)v^r\leq \frac{C_0(\log (2t^*))^{\delta}}{|Q_j^k|}\int_{Q_j^k}\Psi_\varepsilon\left(|f|\right)v^r.\]
	We now assume that $II>1/2$. By performing the same Hölder inequality as in Claim~\ref{af: claim 1}, we get
	\begin{equation}\label{eq: af: Claim 3 - eq1}
    \frac{1}{|Q_j^k|}\int_{Q_j^k}\Phi\left(|f|\right)w_kv^r\leq \frac{v^r(Q_j^k)}{|Q_j^k|}\|\Phi(|f|)\|_{\eta_\varepsilon,Q_j^k,v^r}\|w_k\|_{\tilde\eta_\varepsilon,Q_j^k,v^r}.    
    \end{equation}
    In order to estimate the factor $\|w_k\|_{\tilde\eta_\varepsilon,Q_j^k,v^r}$ we proceed as before. Since
    \begin{align*}
\frac{1}{v^r(Q_j^k)}\int_{Q_j^k\cap\left\{v/a^k>e^{\varepsilon^{-1/(1-\varepsilon)}}\right\}}\frac{v^{r+\varepsilon}}{a^{k\varepsilon}}&\leq  \frac{a^{-k\varepsilon}}{v^r(Q_j^k)}\int_{Q_j^k}v^{r+\varepsilon}\\
&\leq [v^r]_{\text{RH}_{s}}^{s}\frac{a^{-k\varepsilon}|Q_j^k|}{v^r(Q_j^k)}\left(\frac{1}{|Q_j^k|}\int_{Q_j^k}v^r\right)^{s}\\
&= [v^r]_{\text{RH}_{s}}^{s} a^{-k\varepsilon}\left(\frac{1}{|Q_j^k|}\int_{Q_j^k}v^r\right)^{s-1}\\
&\leq [v^r]_{\text{RH}_{s}}^{s} a^{-k\varepsilon}a^{k\varepsilon}\\
&=[v^r]_{\text{RH}_{s}}^{s}   
\end{align*}
we obtain that
\[\|w_k\|_{\tilde\eta_\varepsilon,Q_j^k,v^r}\leq e^{e^{1/(1-\varepsilon)}}+[v^r]_{\text{RH}_{s}}^{s}=C_\varepsilon.\]
By using \eqref{eq: relacion norma con infimo} and choosing $\tau=(2C_\varepsilon)^{-1}$, we can estimate the right-hand side in \eqref{eq: af: Claim 3 - eq1} as follows
\begin{align*}
    \frac{v^r(Q_j^k)}{|Q_j^k|}\|\Phi(|f|)\|_{\eta_\varepsilon,Q_j^k,v^r}\|w_k\|_{\tilde\eta_\varepsilon,Q_j^k,v^r}&\leq \frac{a^{kr}}{2}+\tau\eta_\varepsilon\left(\frac{1}{\tau}\right)\frac{1}{|Q_j^k|}\int_{Q_j^k}\Psi_\varepsilon\left(|f|\right)v^r\\
    &\leq \frac{a^{kr}}{2} + (1+\log(2C_\varepsilon))^{\delta/\varepsilon}\frac{1}{|Q_j^k|}\int_{Q_j^k}\Psi_\varepsilon\left(|f|\right)v^r.
\end{align*}
This yields
\[a^{kr}\leq \frac{C_\varepsilon}{|Q_j^k|}\int_{Q_j^k}\Psi_\varepsilon\left(|f|\right)v^r.\]
This concludes the proof.
	\end{proof}

\section{Applications: Mixed estimates for the generalized fractional integral operator}

Mixed inequalities for the generalized fractional maximal operator $M_{\gamma,\Phi}$ were also given in \cite{B22Pot}. One of the key properties in order to establish the following result was to define an auxiliary operator that is bounded in $L^\infty(uv^r)$ when $v^r\in A_\infty$. This operator is given by 
\[\mathcal{T}_\Phi f(x)=\frac{M_\Phi(fv)(x)}{M_\Phi v(x)}.\]
It is not difficult to see that $M_\Phi v\approx v$ when $\Phi\in\mathfrak{F}_{r,\delta}$ and $v^r$ is an $A_1$-weight, so this operator is an extension of the Sawyer operator $S_\Phi f = M_\Phi(fv)/v$ considered in the main theorem. 

	\begin{coro}[Corrected version of Corollary 3 in \cite{B22Pot}]\label{coro: corolario del teorema principal - 2}
		Let $r\geq 1$, $\delta\geq 0$ and $\Phi\in \mathfrak{F}_{r,\delta}$. Let $u\in A_1$, $v^r\in A_\infty$ and $\Psi$ be a Young function that verifies $\Psi(t) \approx \Phi(t)$, for every $t\geq t_0\geq 0$. Then, there exists $\varepsilon_0>0$ such that the inequality
		\[uv^r\left(\left\{x\in \mathbb{R}^n: \frac{M_\Psi(fv)(x)}{M_\Psi v(x)}>t
		\right\}\right)\leq C_1\int_{\mathbb{R}^n}(\eta_\varepsilon\circ \Psi)\left(\frac{C_2|f(x)|}{t}\right)u(x)v^r(x)\,dx\]
		 holds for every $t>0$ and every $0<\varepsilon<\varepsilon_0$, where $C_1$ depends on $\varepsilon$ and $\eta_\varepsilon(z)=z(1+\log^+z)^{\delta/\varepsilon}$.   
	\end{coro}
	
	\begin{proof}
	     By combining the equivalence between $\Phi$ and $\Psi$ and Proposition~8 in \cite{B22Pot}, we obtain that there exist positive constants $A$ and $B$ such that
	\[A\|\cdot\|_{\Phi,Q}\leq \|\cdot\|_{\Psi,Q}\leq B\|\cdot\|_{\Phi,Q},\]
	for every cube $Q$. By setting $c_1=B/A$, we have that
	\[\frac{M_{\Psi}(fv)(x)}{M_\Psi v(x)} \leq c_1\frac{M_{\Phi}(fv)(x)}{M_\Phi v(x)}\]
	for almost every $x$.
	By applying Corollary~\ref{coro: corolario del teorema principal - 1}, there exists $\varepsilon_0>0$ such that for every $0<\varepsilon<\varepsilon_0$ we have
	\begin{align*}
	uv^r\left(\left\{x\in \mathbb{R}^n: \frac{M_\Psi(fv)(x)}{M_\Psi v(x)}>t
	\right\}\right)&\leq uv^r\left(\left\{x\in \mathbb{R}^n: \frac{M_\Phi(fv)(x)}{M_\Phi v(x)}>\frac{t}{c_1}
	\right\}\right)\\
	&\leq C \int_{\mathbb{R}^n}(\eta_\varepsilon\circ \Phi)\left(\frac{c_1|f|}{t}\right)uv^r.
	\end{align*}
	Observe that
	\[\|\mathcal{T}_\Psi f\|_{L^{\infty}}=\left\|\frac{M_\Psi (fv)}{M_\Psi v}\right\|_{L^{\infty}}\leq  \|f\|_{L^{\infty}},\]
	which directly implies $\|\mathcal{T}_\Psi f\|_{L^{\infty}(uv^r)}\leq \|f\|_{L^{\infty}(uv^r)}$ since the measure given by $d\mu(x)=u(x)v^r(x)\,dx$ is absolutely continuous with respect to the Lebesgue measure. We now apply Lemma~13 in \cite{B22Pot} to obtain
	\begin{align*}
	uv^r\left(\left\{x\in \mathbb{R}^n: \frac{M_\Psi(fv)(x)}{M_\Psi v(x)}>t
	\right\}\right)&\leq C\int_{\{x: |f(x)|>t/2\}}(\eta_\varepsilon\circ \Phi)\left(\frac{2c_1|f(x)|}{t}\right)u(x)v^r(x)\,dx\\
	&\leq C(\eta_\varepsilon\circ \Phi)\left(\frac{c_1}{t_0}\right)\int_{\{x: |f(x)|>t/2\}}(\eta_\varepsilon\circ \Phi)\left(\frac{2t_0|f(x)|}{t}\right)u(x)v^r(x)\,dx\\
	&\leq C_1\int_{\{x: |f(x)|>t/2\}}(\eta_\varepsilon\circ \Psi)\left(\frac{2t_0|f(x)|}{t}\right)u(x)v^r(x)\,dx\\
	&\leq C_1\int_{\mathbb{R}^n}(\eta_\varepsilon\circ \Psi)\left(\frac{C_2|f(x)|}{t}\right)u(x)v^r(x)\,dx.\qedhere
	\end{align*}
	\end{proof}
	
	The corollary above is key in order to obtain mixed inequalities for the generalized fractional maximal operator defined, for $0<\gamma<n$ and a Young function $\Phi$, by the expression
	\[M_{\gamma,\Phi}f(x)=\sup_{Q\ni x}|Q|^{\gamma/n}\|f\|_{\Phi,Q}.\]

    Mixed estimates for this operator are contained in the following theorems.
    
    \begin{teo}	\label{teo: mixta para Mgamma,Phi, caso r<p<n/gamma}
	Let $\Phi(z)=z^r(1+\log^+ z)^\delta$, with $r\geq 1$ and $\delta \geq 0$. Let $0<\gamma<n/r$, $r<p<n/\gamma$ and $1/q=1/p-\gamma/n$. If $u\in A_1$ and $v^{q(1/p+1/r')}\in A_\infty$, then the inequality
	\[uv^{q(1/p+1/r')}\left(\left\{x\in \mathbb{R}^n: \frac{M_{\gamma,\Phi}(fv)(x)}{M_{\varphi} v(x)}>t\right\}\right)^{1/q}\leq C\left[ \int_{\mathbb{R}^n}\left(\frac{|f(x)|}{t}\right)^pu^{p/q}(x)(v(x))^{1+p/r'}\,dx\right]^{1/p},\]
	holds for every positive $t$, where $\varphi(z)=z^{q/p+q/r'}(1+\log^+ z)^{n\delta/(n-r\gamma)}$.
\end{teo}

\begin{proof}
We shall follow the scheme given in the proof of Theorem 4 in \cite{B22Pot}. The difference lies when we apply Corollary~\ref{coro: corolario del teorema principal - 2}, but the function controlling the right-hand side turns out to be auxiliary.
We define
	\[\sigma=\frac{nr}{n-r\gamma}, \quad \nu=\frac{n\delta}{n-r\gamma}, \quad \beta=\frac{q}{\sigma}\left(\frac{1}{p}+\frac{1}{r'}\right),\]
	and let $\xi$ be the auxiliary function given by
	\[\xi(z)=\left\{\begin{array}{ccr}
	z^{q/\beta},&\textrm{ if } & 0\leq z\leq 1,\\
	z^\sigma(1+\log^+z)^\nu, & \textrm{ if } &z> 1.
	\end{array}\right.\]
	Observe that
	\[\xi^{-1}(z)z^{\gamma/n}\approx \frac{z^{1/\sigma+\gamma/n}}{(1+\log^+z)^{\nu/\sigma}}=\frac{z^{1/r}}{(1+\log^+z)^{\delta/r}}\lesssim \Phi^{-1}(z),\]
	for every $z\geq 1$.
	Observe that $\beta>1$: indeed, since $p>r$ we have $q>\sigma$ and thus $q/(\sigma r')>1/r'$. On the other hand, $q/(p\sigma)>1/r$. By combining these two inequalities we have $\beta>1$. Applying Proposition 10 and Lemma 9 with $\beta$ from \cite{B22Pot},  we can conclude that
	\begin{equation}\label{eq: teo - mixta para M_{gamma,Phi}, caso r<p<n/gamma - eq1}
	M_{\gamma,\Phi}\left(\frac{f_0}{w}\right)(x)\leq C\left[M_\xi\left(\frac{f_0^{p\beta/q}}{w^{\beta}}\right)(x)\right]^{1/\beta}\left(\int_{\mathbb{R}^n}f_0^p(y)\,dy\right)^{\gamma/n}.
	\end{equation}
	
	Also observe that 
	\begin{equation}\label{eq: teo - mixta para M_{gamma,Phi}, caso r<p<n/gamma - eq2}
	\left(M_\xi v^\beta(x)\right)^{1/\beta}\lesssim M_{\varphi} v(x), \quad \textrm{ a.e. }x.
	\end{equation}
	
	Notice that $\xi$ is equivalent to a Young function in $\mathfrak{F}_{\sigma,\nu}$ for $t\geq 1$. Since $q(1/p+1/r')=\beta\sigma$, if we set $f_0=|f|wv$, then we can use inequalities \eqref{eq: teo - mixta para M_{gamma,Phi}, caso r<p<n/gamma - eq1} and \eqref{eq: teo - mixta para M_{gamma,Phi}, caso r<p<n/gamma - eq2} to estimate
	\begin{align*}
	uv^{\tfrac{q}{p}+\tfrac{q}{r'}}\left(\left\{x: \frac{M_{\gamma,\Phi}(fv)(x)}{M_{\varphi} v(x)}>t\right\}\right)&\lesssim uv^{\beta\sigma}\left(\left\{x: \frac{M_{\gamma,\Phi}(fv)(x)}{\left(M_\xi v^\beta(x)\right)^{1/\beta}}>t\right\}\right)\\
	&\leq uv^{\beta\sigma}\left(\left\{x: \frac{M_\xi\left(f_0^{p\beta/q}w^{-\beta}\right)(x)}{M_\xi v^\beta(x)}>\frac{t^\beta}{\left(\int|f_0|^p\right)^{\beta\gamma/n}}\right\}\right).
	\end{align*}
	Since $v^{\beta\sigma}\in A_\infty$, by Corollary~\ref{coro: corolario del teorema principal - 2} there exists $\varepsilon_0>0$ such that the inequality
	\[uv^{\beta\sigma}\left(\left\{x: \frac{M_\xi\left(f_0^{p\beta/q}w^{-\beta}\right)(x)}{M_\xi v^\beta(x)}>t_0\right\}\right)\leq  C_\varepsilon\int_{\mathbb{R}^n}(\eta_{\varepsilon}\circ\xi)\left(c\frac{f_0^{p\beta/q}w^{-q}}{t_0}\right)uv^{\beta\sigma}\]
	holds for every $0<\varepsilon<\varepsilon_0$, with $t_0=t^\beta\|f_0\|_{p}^{-p\beta\gamma/n}$. Notice that $\eta_\varepsilon(z)=z(1+\log^+z)^{\nu/\varepsilon}$ in this case. Fixed $\varepsilon$, we write
	\[\int_{\mathbb{R}^n}(\eta_\varepsilon\circ\xi)\left(c\frac{|f|^{p\beta/q}(wv)^{\beta(p/q-1)}}{t^\beta}\left[\int_{\mathbb{R}^n}|f|^p(wv)^p\right]^{\gamma/n\beta}\right)uv^{\sigma\beta}
	=\int_{\mathbb{R}^n}(\eta_\varepsilon\circ\xi)(\lambda)uv^{\sigma\beta},\]
	where
	\[\lambda=c\frac{|f|^{p\beta/q}(wv)^{\beta(p/q-1)}}{t^\beta}\left[\int_{\mathbb{R}^n}|f|^p(wv)^p\right]^{\gamma/n\beta}.\]
	We further split $\mathbb{R}^n$ into the sets
	$A=\{x\in \mathbb{R}^n: \lambda(x)\leq 1\}$ and $B=\mathbb{R}^n\backslash A$. Since $(\eta_{\varepsilon} \circ \xi)(z)=z^{q/\beta}$ for $0\leq z\leq 1$, we have that
	\[\int_{A}(\eta_\varepsilon\circ \xi)(\lambda(x))u(x)[v(x)]^{\sigma\beta}\,dx=\int_A [\lambda(x)]^{q/\beta}u(x)[v(x)]^{\sigma\beta}\,dx.\]
	If we set $w=u^{1/q}v^{1/p+1/r'-1}$, then 
	\begin{align*}
	\lambda^{q/\beta}uv^{\sigma\beta}&=c^{q/\beta}\frac{|f|^p}{t^q}(wv)^{p-q}\left[\int_{\mathbb{R}^n}|f|^p(wv)^p\right]^{q\gamma/n}uv^{\sigma\beta}\\
	&=c^{q/\beta}\frac{|f|^p}{t^q}\left[\int_{\mathbb{R}^n}|f|^p(wv)^p\right]^{q\gamma/n}u^{p/q}v^{\sigma\beta+(p-q)(1/p+1/r')}.
	\end{align*}
	Observe that
	\[\sigma\beta+(p-q)\left(\frac{1}{p}+\frac{1}{r'}\right)=q\left(\frac{1}{p}+\frac{1}{r'}\right)+(p-q)\left(\frac{1}{p}+\frac{1}{r'}\right)=1+\frac{p}{r'}.\]
	Also, notice that
	\[(wv)^p=u^{p/q}v^{1+p/r'-p+p}=u^{p/q}v^{1+p/r'}.\]
	Therefore,
	\begin{align*}
	\int_{A}(\eta_\varepsilon\circ\xi)(\lambda)uv^{\sigma\beta}&\leq \frac{c^{q/\beta}}{t^q}\left[\int_{\mathbb{R}^n}|f|^pu^{p/q}v^{1+p/r'}\right]^{q\gamma/n}\left[\int_{\mathbb{R}^n} |f|^pu^{p/q}v^{1+p/r'}\right]\\
	&= \frac{c^{q/\beta}}{t^q}\left[\int_{\mathbb{R}^n}|f|^pu^{p/q}v^{1+p/r'}\right]^{1+q\gamma/n}\\
	&=\frac{c^{q/\beta}}{t^q}\left[\int_{\mathbb{R}^n}|f|^pu^{p/q}v^{1+p/r'}\right]^{q/p}.
	\end{align*}
	On the other hand, $\lambda(x)>1$ over $B$ and 
	\[(\eta_\varepsilon\circ \xi)(z)\lesssim z^\sigma(1+\log z)^{\nu(1+1/\varepsilon)},\]
	and this function has an upper type $q/\beta$. Therefore we can estimate the integrand by $\lambda^{q/\beta}uv^{\sigma\beta}$ and proceed as we did with the set $A$. 
	Thus, we obtain 
	\[uv^{q(1/p+1/r')}\left(\left\{x\in \mathbb{R}^n: \frac{M_{\gamma,\Phi}(fv)(x)}{M_\varphi v(x)}>t\right\}\right)^{1/q}\leq C\left[ \int_{\mathbb{R}^n}\left(\frac{|f|}{t}\right)^pu^{p/q}v^{1+p/r'}\right]^{1/p}.\qedhere\]
\end{proof}

\begin{teo}[Corrected version of Theorem 5 in \cite{B22Pot}]\label{teo: mixta para Mgamma,Phi, caso p=r}
	Let $\Phi(z)=z^r(1+\log^+z)^\delta$, with $r\geq 1$ and $\delta\geq 0$. Let $0<\gamma<n/r$ and $1/q=1/r-\gamma/n$. If $u\in A_1$ and $v^q\in A_\infty$, then there exists a positive constant $\varepsilon_0$ such that the inequality
	\[uv^q\left(\left\{x\in \mathbb{R}^n: \frac{M_{\gamma,\Phi}(fv)(x)}{v(x)}>t\right\}\right)\leq C\, \varphi_\varepsilon\left(\int_{\mathbb{R}^n}\Phi_{\gamma,\varepsilon}\left(\frac{|f(x)|}{t}\right)\Psi\left(u^{1/q}(x)v(x)\right)\,dx\right),\]
	holds for $0<\varepsilon<\varepsilon_0$,
	where $\varphi_\varepsilon(z)=[z(1+\log^+z)^{\delta(1+1/\varepsilon)}]^{q/r}$, $\Psi_\varepsilon(z)=z^r(1+\log^+(z^{1-q/r}))^{q\delta(1+1/\varepsilon)/r}$, $\Phi_{\gamma,\varepsilon}(z)=\Phi(z)(1+\log^+z)^{\delta(1+1/\varepsilon) q\gamma/n+\delta/\varepsilon}$ and $C$ depends on $\varepsilon$. 
\end{teo}

\begin{proof}
	Set $\xi(z)=z^q(1+\log^+z)^\nu$, where $\nu=\delta q/r$. Thus $z^{\gamma/n}\xi^{-1}(z)\lesssim \Phi^{-1}(z)$. By applying Proposition 10 in \cite{B22Pot} with $p=r$ we have that
	\[M_{\gamma,\Phi}\left(\frac{f_0}{w}\right)(x)\leq C\left[M_\xi\left(\frac{f_0^{r/q}}{w}\right)\right](x)\left(\int_{\mathbb{R}^n}f_0^r(y)\,dy\right)^{\gamma/n}.\]
	By setting $f_0=|f|wv$ we can write
	\begin{align*}
	uv^{q}\left(\left\{x: \frac{M_{\gamma,\Phi}(fv)(x)}{v(x)}>t\right\}\right)&=uv^{q}\left(\left\{x: \frac{M_{\gamma,\Phi}(f_0/w)(x)}{v(x)}>t\right\}\right)\\
	&\leq uv^{q}\left(\left\{x: \frac{M_{\xi}(f_0^{r/q}/w)(x)}{M_\xi v(x)}>\frac{t}{\left(\int f_0^r \right)^{\gamma/n}}\right\}\right).
	\end{align*}
	Since $\xi\in \mathfrak{F}_{q,\nu}$, by Corollary~\ref{coro: corolario del teorema principal - 1} there exists $\varepsilon_0>0$ such that
	\begin{equation}\label{eq: eq1 - teo mixta para M_{gamma,Phi}, caso p=r}
	uv^{q}\left(\left\{x: \frac{M_{\gamma,\Phi}(fv)(x)}{v(x)}>t\right\}\right)\leq C_\varepsilon\int_{\mathbb{R}^n}(\eta_{\varepsilon}\circ\xi)\left(\frac{f_0^{r/q}\left(\int f_0^r\right)^{\gamma/n}}{wv t}\right)uv^q,
	\end{equation}
	for $0<\varepsilon<\varepsilon_0$ and being $\eta_\varepsilon(z)=z(1+\log^+z)^{\nu/\varepsilon}$. Fixed $\varepsilon$,
	the argument of $\eta_\varepsilon \circ \xi$ above can be written as
	\begin{align*}
	\frac{f_0^{r/q}\left(\int f_0^r\right)^{\gamma/n}}{wv t}&=\left(\frac{|f|}{t}\right)^{r/q}(wv)^{r/q-1}\left(\int_{\mathbb{R}^n} \left(\frac{|f|}{t}\right)^r(wv)^r\right)^{\gamma/n}\\
	&=\left[\left(\frac{|f|}{t}\right)(wv)^{1-q/r}\left(\int_{\mathbb{R}^n} \left(\frac{|f|}{t}\right)^r(wv)^r\right)^{\gamma q/(nr)}\right]^{r/q}.
	\end{align*}
	Observe that for $0\leq z\leq 1$, $(\eta_\varepsilon \circ\xi)(z^{r/q})=z^r$, and for $z>1$ we have
	\[(\eta_\varepsilon \circ \xi)(z^{r/q})\lesssim z^r(1+\log z)^{\nu(1+1/\varepsilon)},\]
	which implies that $(\eta_\varepsilon \circ \xi)(z^{r/q})\lesssim \Phi_{\gamma,\varepsilon}(z)=z^r(1+\log^+z)^{\nu(1+1/\varepsilon)}$, for every $z\geq 0$. Since $\Phi_{\gamma,\varepsilon}$ is submultiplicative, we can estimate as follows
	\begin{align*}
	(\eta_\varepsilon \circ \xi )\left(\frac{f_0^{r/q}\left(\int_{\mathbb{R}^n}f_0^r\right)^{\gamma/n}}{wv t}\right)&\leq \Phi_{\gamma,\varepsilon}\left(\left(\frac{|f|}{t}\right)(wv)^{1-q/r}\left(\int_{\mathbb{R}^n} \left(\frac{|f|}{t}\right)^r(wv)^r\right)^{\gamma q/(nr)}\right)\\
	&\leq \Phi_{\gamma,\varepsilon}\left(\left[\int_{\mathbb{R}^n}\Phi_{\gamma,\varepsilon}\left(\frac{|f|}{t}\right)(wv)^r\right]^{\gamma q/(nr)}\right)\Phi_{\gamma,\varepsilon}\left(\frac{|f|}{t}(wv)^{1-q/r}\right)
	\end{align*}
	Returning to \eqref{eq: eq1 - teo mixta para M_{gamma,Phi}, caso p=r} and setting $w=u^{1/q}$, the right hand side is bounded by
	\[ \Phi_{\gamma,\varepsilon}\left(\left[\int_{\mathbb{R}^n}\Phi_{\gamma,\varepsilon}\left(\frac{|f|}{t}\right)(wv)^r\right]^{\gamma q/(nr)}\right)\int_{\mathbb{R}^n} \Phi_{\gamma,\varepsilon}\left(\frac{|f|}{t}(wv)^{1-q/r}\right)(wv)^q.\]
	Notice that $\Phi_{\gamma,\varepsilon}(z^{1-q/r})z^q\leq \Psi_\varepsilon(z)$. Therefore, the expression above is bounded by 
	\[\Phi_{\gamma,\varepsilon}\left(\left[\int_{\mathbb{R}^n}\Phi_{\gamma,\varepsilon}\left(\frac{|f|}{t}\right)\Psi_\varepsilon(u^{1/q}v)\right]^{\gamma q/(nr)}\right)\int_{\mathbb{R}^n} \Phi_{\gamma,\varepsilon}\left(\frac{|f|}{t}\right)\Psi_\varepsilon(u^{1/q}v).\]
	To finish, observe that 
	\[z\Phi_{\gamma,\varepsilon}(z^{\gamma q/(nr)})\lesssim z^{1+\gamma q/n}(1+\log^+ z)^{\nu(1+1/\varepsilon)}= z^{q/r}(1+\log^+ z)^{\delta q(1+1/\varepsilon)/r}=\varphi_\varepsilon(z).\qedhere\]
\end{proof}

\def\cprime{$'$}
\providecommand{\bysame}{\leavevmode\hbox to3em{\hrulefill}\thinspace}
\providecommand{\MR}{\relax\ifhmode\unskip\space\fi MR }
\providecommand{\MRhref}[2]{%
  \href{http://www.ams.org/mathscinet-getitem?mr=#1}{#2}
}
\providecommand{\href}[2]{#2}


\begin{thebibliography}{1}

\bibitem{B22Pot}
Fabio Berra, \emph{From {$A_1$} to {$A_\infty$}: new mixed inequalities for
  certain maximal operators}, Potential Anal. \textbf{57} (2022), no.~1, 1--27.

\bibitem{Berra-Carena-Pradolini(MN)}
Fabio Berra, Marilina Carena, and Gladis Pradolini, \emph{Improvements on
  {S}awyer type estimates for generalized maximal functions}, Math. Nachr.
  \textbf{293} (2020), no.~10, 1911--1930.

\bibitem{HP13}
Tuomas Hyt\"{o}nen and Carlos P\'{e}rez, \emph{Sharp weighted bounds involving
  {$A_\infty$}}, Anal. PDE \textbf{6} (2013), no.~4, 777--818.

\bibitem{L-O-P}
K.~Li, S.~Ombrosi, and C.~P\'{e}rez, \emph{Proof of an extension of {E}.
  {S}awyer's conjecture about weighted mixed weak-type estimates}, Math. Ann.
  \textbf{374} (2019), no.~1-2, 907--929.

\end{thebibliography}
\end{document}